\documentclass[a4paper, 11 pt] {article}

\usepackage{amsmath}
\usepackage{amsfonts}
\usepackage{amsthm}
\usepackage{amscd}

\usepackage{graphicx}
\usepackage{epsfig}
\usepackage{graphics}

\newcommand\grad{{\bf \nabla}}
\newcommand\rvec{{\bf r}}
\newcommand\la{{\lambda}}
\newcommand\nvec{{\bf n}}
\newcommand\evec{{\bf e}}
\newcommand\mvec{{\bf m}}
\newcommand\pvec{{\bf p}}
\newcommand\eps{{\epsilon}}
\newcommand\Qvec{{\bf Q}}
\newcommand{\Acal}{{\cal A}}
\newcommand{\Rr}{{\mathbb R}}

\newtheorem{thm}{Theorem}
\newtheorem{lem}{Lemma}
\newtheorem{prop}{Proposition}

\title{Equilibrium order parameters of nematic liquid crystals in the Landau-De Gennes theory }

\author{Apala Majumdar\thanks{Mathematical Institute, University of Oxford,
        24--29 St.Giles, Oxford, OX1 3LB ({\tt majumdar@maths.ox.ac.uk}).}
        }

\begin{document}

\maketitle

 \begin{abstract} We study equilibrium liquid crystal
 configurations in three-dimensional domains, within the continuum
 Landau-De Gennes theory. We obtain explicit bounds for the
 equilibrium scalar order parameters in terms of the temperature
 and material-dependent constants. We explicitly quantify the
 temperature regimes where the Landau-De Gennes predictions match
 and the temperature regimes where the Landau-De Gennes
 predictions don't match the probabilistic second-moment definition of the $\Qvec$-tensor
 order parameter. The regime of agreement may be interpreted as
 the regime of validity of the Landau-De Gennes theory since the Landau-De Gennes theory 
predicts large values of the equilibrium scalar order parameters - larger than unity, in the low-temperature regime. We discuss
 a modified Landau-De Gennes energy functional which yields
 physically realistic values of the equilibrium scalar order parameters in all temperature regimes.
\end{abstract}

\textbf{Keywords:} Nematic liquid crystals, order parameters,
equilibrium configurations, Landau--De Gennes theory

\textbf{AMS Classification:
} 35Qxx, 35Jxx, 35Bxx

\section{Introduction}
 Liquid crystals are an intermediate phase of matter between the
 commonly observed solid and liquid phases \cite{dg}. In the
 simplest liquid crystal phase, the nematic phase, the constituent
 rod--like molecules translate freely as in a conventional liquid
 but whilst flowing, tend to align along certain locally preferred directions i.e. they
 exhibit a certain degree of long-range orientational ordering.
Liquid crystals have attracted a lot of interest in recent years
because of their unique physical properties and continue to do so
because of their diverse applications \cite{kg} and their
analogies to other physical systems.

The mathematical theory of nematic liquid crystals is very rich;
for a review see \cite{straley, Lin}. The key ingredient of any
mathematical theory for nematic liquid crystals is the definition
of an \emph{order parameter} that distinguishes the ordered
nematic phase from the disordered isotropic liquid phase.
Mean-field liquid crystal theories, such as the Maier-Saupe
theory, describe the liquid crystal configuration in terms of a
probability distribution function $\psi$ on the unit sphere. The
order-parameter, known as the $\Qvec$-tensor order parameter, is
defined in terms of the second moment of $\psi$ \cite{dg,
newtonmottram}. This \emph{probabilistic second-moment definition}
naturally requires $\Qvec$ to be a symmetric, traceless $3\times
3$ matrix and imposes certain constraints on its eigenvalues,
which represent the degree of ordering. The Landau-De Gennes
theory, on the other hand, is a continuum theory for nematic
liquid crystals and does not contain any information about either
$\psi$ or the intermolecular interactions \cite{dg,
newtonmottram}. The $\Qvec$-tensor order parameter, within the
Landau-De Gennes framework, is a symmetric, traceless $3\times 3$
matrix with no a priori constraints on the eigenvalues. The
Landau-De Gennes energy functional is a nonlinear integral
functional of $\Qvec$ and its spatial derivatives and the
equilibrium, physically observable configurations correspond to
either global or local minimizers of this energy subject to the
imposed boundary conditions.

 A natural question of interest is - do the
 equilibrium configurations predicted by the Landau-De Gennes theory
 agree with the probabilistic second-moment definition of $\Qvec$? We systematically address
  this question in this paper. We obtain explicit bounds for the scalar order parameters
  of global energy minimizers, referred to as \emph{equilibrium scalar order parameters}, in
  terms of the temperature and the material-dependent constants. These bounds quantify
  (to some extent) the competing effects of the different terms in the Landau-De Gennes
  energy density. Further, these bounds are compared to the probabilistic second-moment
  definition of $\Qvec$. This allows us to explicitly delineate the regions of agreement
   and the regions of disagreement and we find that the Landau-De Gennes predictions don't
   match the probabilistic second-moment definition in the low-temperature regime. In particular, the equilibrium scalar order parameters, within the Landau-De Gennes framework, can take physically unrealistic values (larger than unity) in the low-temperature regime. Our results largely depend on the use of maximum principle type of arguments for
    nonlinear elliptic systems of partial differential equations and they can be
    readily extended to more general energy functionals than the ones considered in this paper.

The derivation of the Landau-De Gennes energy density is valid
near the isotropic state, close to the nematic-isotropic
transition temperature. Therefore, it is well-expected that the
predictions are physically unrealistic in the low temperature
regime. However, our results show that the Landau-De Gennes
predictions fail to be consistent with the probabilistic
second-moment definition within a sufficiently small neighbourhood
of the nematic-isotropic transition temperature. In principle, one
would want to develop a continuum theory that works for all
temperature regimes. In the last part of the paper, we briefly
outline a Ginzburg-Landau approach that remedies the flawed
predictions in the low temperature regimes. We define a modified
Landau-De Gennes energy functional such that the energy density
blows up whenever the liquid crystal configuration violates the
constraints imposed by the probabilistic second-moment definition
of $\Qvec$ or equivalently whenever the scalar order parameters are physically unrealistic. One deficiency of this approach is that it has no
apparent connection with the mean-field microscopic approaches. A
different approach has been suggested in \cite{luckhurst} and we
hope to systematically investigate a microscopic-macroscopic
derivation of a continuum energy functional in future work
\cite{jmbam}.

 The paper is organized as follows. In Section~\ref{sec:2}, we review the
 probabilistic second-moment definition of the order parameter. In Section~\ref{sec:3},
 we study equilibrium liquid crystal configurations within the continuum
 Landau--De Gennes theory. In Section~3.1,
 we consider spatially homogeneous cases whereas in Sections~3.2 and 3.3, we
 include spatial inhomogeneities into the model and obtain upper bounds for the
 corresponding equilibrium scalar order parameters. These bounds explicitly
 define the domain for the equilibrium scalar order parameters in terms of the
 temperature and the material-dependent parameters. In Section~\ref{sec:4},
 we discuss the main results and conclusions of this paper and suggest
 future research directions.

 \section{The Probabilistic Second-Moment Definition}
 \label{sec:2}

In this section, we briefly review the probabilistic second-moment
definition of the $\Qvec$-tensor order parameter and the
Maier-Saupe mean-field liquid crystal theory. The interested
reader is referred to \cite{dg, ms, virga} for details and we
present the main points here for completeness.

Within the simplest microscopic model, the nematic molecules are
modelled by elongated rods where the long molecular axes tend to
align along certain locally preferred directions \cite{dg,
newtonmottram}. The state of alignment of the nematic molecules is
described by a probability distribution function for the molecular
orientations on the unit sphere, $\psi: S^2 \rightarrow
\mathbb{R}^+$, since $S^2\subset \mathbb{R}^3$ is the space of all
admissible directions. The probability distribution function,
$\psi(\pvec)$, gives the probability of finding molecules oriented
in the direction $\pvec \in S^2$. Then $\psi$ has the following
properties - \cite{dg,ball} -
\begin{eqnarray} && \psi(\pvec) \geq 0 \quad \pvec\in S^2 \nonumber \\ &&
\psi\left(\pvec\right) = \psi\left(-\pvec \right) \label{eq:3a} \\
&& \int_{S^2} \psi\left(\pvec\right)~d\pvec = 1. \label{eq:3}
\end{eqnarray} where (\ref{eq:3a}) accounts for the indistinguishability of the states $\pvec$ and $-\pvec$ on the unit sphere.

The macroscopic variables are defined in terms of the moments of
$\psi$. The first moment vanishes because of the equivalence
between antipodal points, $\pvec\equiv-\pvec$. We define the
nematic order parameter, the $\Qvec$-tensor order parameter, to be
the normalized second moment of the probability distribution
function as follows \cite{dg, newtonmottram}-
\begin{equation}
\Qvec = \int_{S^2}\left(\pvec \otimes \pvec -
\frac{1}{3}\mathbf{I}\right)~\psi(\pvec)~d\pvec \label{eq:5}
\end{equation} We refer to (\ref{eq:5}) as the probabilistic second-moment definition of $\Qvec$ in the rest of the paper. For an isotropic system, where all directions in space are equally likely, the function $\psi$ is a constant i.e.
\begin{equation}
\psi(\pvec) = \frac{1}{4\pi} \quad \forall\pvec\in S^2
\label{eq:4}
\end{equation} and consequently, $\Qvec=0$. On the other hand, for a perfectly aligned system where the nematic molecules identically align along a pair of unit-vectors $\left(\evec, -\evec\right)$, the function $\psi$ is given by -
\begin{equation}
\psi(\pvec) = \frac{1}{8\pi} \left(\delta_{S^2}\left(\evec,\pvec\right) +
\delta_{S^2}\left(-\evec,\pvec\right)\right) \label{eq:stat1}
\end{equation}where $\delta_{S^2}$ is the Dirac-delta function on $S^2$ and the corresponding $\Qvec$-tensor is $\Qvec=\left(\evec\otimes\evec - \frac{1}{3}\mathbf{I}\right)$.

It follows directly from (\ref{eq:5}) that $\Qvec$ is a
symmetric, traceless $3\times 3$ matrix. From the spectral
decomposition theorem, we can express $\Qvec$ in terms of a triad
of orthonormal eigenvectors, $\left\{\evec_1, \evec_2,
\evec_3\right\}$, and corresponding eigenvalues,
$\left\{\la_1,\la_2,\la_3\right\}$, subject to the tracelessness
condition $\sum_i \la_i = 0$.
\begin{equation}
\Qvec = \la_1 \evec_1 \otimes \evec_1 + \la_2 \evec_2 \otimes
\evec_2 +\la_3 \evec_3 \otimes \evec_3 \quad \textrm{where $\sum_i
\lambda_i = 0$.} \label{eq:6}
\end{equation}
Nematic liquid crystals are broadly classified into three main
families according to the eigenvalue structure of $\Qvec$. A
nematic liquid crystal is called isotropic when it has three equal
eigenvalues (the tracelessness condition implies that $\Qvec=0$),
uniaxial when it has a pair of equal non-zero eigenvalues and
biaxial when it has three distinct eigenvalues \cite{dg,
newtonmottram}. The eigenvalues measure the degree of
orientational ordering along the corresponding eigenvectors and
one can verify that the eigenvalues are constrained by the
following inequalities -
\begin{equation}
-\frac{1}{3} \leq \lambda_i =
\int_{S^2}\left(\pvec\cdot\evec_i\right)^2 \psi(\pvec)~d\pvec -
\frac{1}{3} \leq \frac{2}{3},~ \quad ~i = 1\ldots 3 \label{eq:7}
\end{equation} since
$0\leq
\int_{S^2}\left(\pvec\cdot\evec_i\right)^2~\psi(\pvec)~d\pvec\leq
1$. If the eigenvalue $\la_i=-\frac{1}{3}$ (the lower bound in
(\ref{eq:7})), then the function $\psi$ is supported on the great
circle perpendicular to the corresponding eigenvector $\evec_i$.
On the other hand, if $\la_i=\frac{2}{3}$ (the upper bound in
(\ref{eq:7})), then $\psi$ is as in (\ref{eq:stat1}) and the
liquid crystal molecules line up perfectly along the pair of
unit-vectors $\left(\evec_i, -\evec_i\right)$. For example, the
liquid crystal state, $\left(\la_1,\la_2,\la_3\right) =
\left(\frac{2}{3}, -\frac{1}{3},-\frac{1}{3}\right)$, is an
example of a perfectly ordered state along the eigenvector
$\evec_1$ and exhibits \emph{prolate} uniaxial symmetry whereas
the liquid crystal state, $\left(\la_1,\la_2,\la_3\right) =
\left(-\frac{1}{3}, \frac{1}{6},\frac{1}{6}\right)$, has the
molecules aligned in the plane orthogonal to $\evec_1$ and
exhibits \emph{oblate} uniaxial symmetry \cite{forest1}. From a
physical point of view, the limiting values, $\la_i =
-\frac{1}{3}$ or $\la_i=\frac{2}{3}$, represent unrealistic
configurations.

The $\Qvec$-tensor order parameter can be expressed more concisely
in terms of just a pair of eigenvectors
$\left(\evec_1,\evec_2\right)$ and a pair of scalar order
parameters $\left(s,r \right)$ as shown below
\cite{newtonmottram}.
\begin{equation}
\Qvec = s \left(\evec_1 \otimes \evec_1 -
\frac{1}{3}\mathbf{I}\right) + r\left(\evec_2\otimes\evec_2 -
\frac{1}{3}\mathbf{I}\right), \label{eq:8}
\end{equation} where $s,r$ are linear combinations of the
$\la_i$'s given by
\begin{eqnarray}&&
s = \la_1 - \la_3 = 2\lambda_1 + \la_2 \nonumber \\ && r = \la_2 -
\la_3 = \la_1 + 2\la_2 . \label{eq:9}
\end{eqnarray} The constraints (\ref{eq:7}) directly translate
into constraints for the scalar order parameters $(s,r)$ in
(\ref{eq:9}) and necessarily imply that $(s,~r)$ take values
inside or on the boundary of the \emph{physical triangle},
$T_\psi$, illustrated in Figure~\ref{fig:1}. On each of the
boundary segments of $T_\psi$, one of the eigenvalues $\la_i$
necessarily attains the lower bound in (\ref{eq:7}). For example,
on the boundary segment $s+r=1$, we have $\la_3 = - \frac{1}{3}$.
Similarly, every vertex of $T_\psi$ represents a state of perfect
alignment along of the eigenvectors of $\Qvec$. For example, the
vertex $(s,~r) = (1,0)$ represents a state of perfect alignment
along the eigenvector $\evec_1$. We call $T_\psi$, the physical triangle, on the grounds that the scalar order parameters are appropriately bounded (less than unity) inside $T_\psi$ and the boundary represents physically unrealistic liquid crystal configurations.

For definiteness, we can assume a
specific ordering of the eigenvalues such as $\la_3 \leq \la_2
\leq \la_1$. Then $\la_1$ is necessarily non-negative and $\la_3$
is necessarily non-positive and the constraints (\ref{eq:7})
require $(s,~r)$ to take values inside a subset of $T_\psi$, which
is referred to as a \emph{fundamental domain} $T_f$ defined below
\begin{equation}
T_f = \left\{ (s,~r); 0\leq s\leq 1,~ 0\leq r\leq \min\left\{s,
1-s\right\} \right\} \subset T_\psi. \label{eq:fd}
\end{equation}
Analogous remarks apply to the other five possibilities for the
ordering of the eigenvalues.

The Maier-Saupe theory is a mean-field theory for uniaxial nematic
liquid crystals \cite{dg, palffy}. The Maier-Saupe free energy has
two contributions -
\begin{equation}
I_{MS}[\psi] = \int_{S^2}\psi(\pvec) \log \psi(\pvec)~d\pvec -
\frac{1}{2}U(T)S^2 \label{eq:maier-saupe}
\end{equation} where $U(T)$ accounts for the intermolecular
interactions and is temperature-dependent and $S$ is the uniaxial
scalar order parameter. A standard minimization procedure for
$I_{MS}$ yields a self-consistent equation for the equilibrium
order parameter, $S(T)$, as a function of the temperature. For
high temperatures, the isotropic phase $S=0$ is the global energy
minimizer whereas for temperatures below a certain critical
temperature $T_c$, the nematic phase is globally stable and the
Maier-Saupe theory predicts a first-order nematic-isotropic phase
transition at the critical temperature $T_c$.

\begin{figure}
[hp]

\begin{center}
\includegraphics[width=3 in, height=3 in]{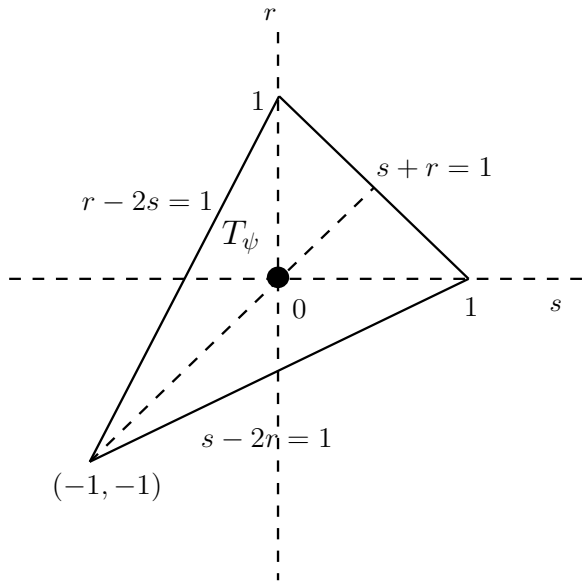}
\caption{The physical triangle $T_{\psi}$.
The origin $(s,~r)=(0,~0)$ represents the isotropic liquid state; the dotted lines $U = \left\{ (s,r)\in T_{\psi}:\textrm{$s=0$ or $r=0$ or $s=r$}\right\}\setminus (0,0)$ represent uniaxial states and
$B = T_{\psi}\setminus \left\{U \cup (0,0)\right\}$ is the biaxial region.}\label{fig:1}
\end{center}
\end{figure}

\section{The Landau--De Gennes Theory}
\label{sec:3}

In this section, we study equilibrium nematic configurations
within the continuum Landau-De Gennes theory. The Landau--De
Gennes theory describes the state of a nematic liquid crystal by a
macroscopic order parameter - the $\Qvec$-tensor order parameter,
which is defined in terms of macroscopic quantities such as the
magnetic susceptibility \cite{dg, Lin}. Within the Landau-De
Gennes framework, $\Qvec$ is a symmetric, traceless $3\times 3$
matrix with no a priori bounds on the eigenvalues; in particular,
the eigenvalues are not constrained by the inequalities
(\ref{eq:7}).

We work in a three-dimensional setting, take the domain
$\Omega\subset \Rr^3$ to be bounded and simply-connected with
smooth boundary. Let $S_0$ denote the space of symmetric,
traceless $3\times 3$ matrices
\begin{equation}
S_0= \left\{ \Qvec \in \mathbf{M}^{3\times 3};~
\Qvec_{\alpha\beta} = \Qvec_{\beta\alpha},~ \Qvec_{\alpha\alpha} =
0 \right\} \label{eq:32}.
\end{equation} The corresponding matrix norm is defined to be
\cite{Lin}
\begin{equation}
\label{eq:matrixnorm} |\Qvec|^2 =
\Qvec_{\alpha\beta}\Qvec_{\alpha\beta} \quad \alpha,\beta = 1
\ldots 3 \end{equation} and the Einstein summation convention is
used here and elsewhere in the paper. We define our admissible
space $\Acal$ to be
\begin{equation}
\mathcal{A} = \left\{ \Qvec \in W^{1,2}\left(\Omega, S_0\right);~
\Qvec = \Qvec_0 \quad \textrm{on $\partial\Omega$}\right\}
\label{eq:31}
\end{equation}
where the Sobolev space $W^{1,2}\left(\Omega, S_0\right)$ is given
by \cite{gilbarg, evans}
\begin{equation}
W^{1,2}\left(\Omega, S_0\right) = \left\{ \Qvec\in S_0;~
\int_{\Omega}\left|\Qvec\right|^2 + |\grad \Qvec |^2~dV < \infty
\right\} \label{eq:33}
\end{equation} and $\Qvec_0$ is a smooth, \emph{physically
realistic} boundary condition in the sense that its scalar order
parameters, $(s,~r)$ in (\ref{eq:9}), are inside the physical
triangle $T_{\psi}$.

In the absence of external fields and surface energies, the
Landau-De Gennes energy functional, $I_{LG}$, is given by
\begin{equation} I_{LG}\left[\Qvec\right] = \int_{\Omega}
f_B\left(\Qvec\right) + L|\grad \Qvec|^2~dV. \label{eq:12}
\end{equation}
Here $f_B$ is the bulk energy density, $L>0$ is a
material-dependent elastic constant, $$|\grad \Qvec|^2 =
\Qvec_{ij,k}\Qvec_{ij,k} \quad i,j,k=1\ldots 3 $$ is the elastic
energy density where $\Qvec_{ij,k} = \frac{\partial
\Qvec_{ij}}{\partial \rvec_k}$ denote the first partial
derivatives of $\Qvec$.

\textit{Comment: We work with the simplest form of the elastic
energy density - the one-constant elastic energy density in
(\ref{eq:12}). There are more general forms of the elastic energy
density, see \cite{newtonmottram, gartland}.}

The bulk energy density $f_B$ is a scalar function of $\Qvec$ and
it dictates the preferred liquid crystal phase - isotropic,
uniaxial or biaxial. We work with the simplest form of $f_B$ that
allows for a first-order nematic-isotropic phase transition. This
simplest form of $f_B$ is a quartic polynomial in $\Qvec$ as shown
below
\begin{eqnarray}
&& f_B\left(\Qvec\right) = \frac{a}{2}\textrm{tr}\Qvec^2 -
\frac{b}{3}\textrm{tr} \Qvec^3 + \frac{c}{4}\left(\textrm{tr}
\Qvec^2 \right)^2 ~\textrm{with}\label{eq:13}
\\ && \textrm{tr}\Qvec^2 = \Qvec_{\alpha\beta}\Qvec_{\alpha\beta}
,~ \textrm{tr}\Qvec^3 = \Qvec_{\alpha\beta}\Qvec_{\beta
\gamma}\Qvec_{\gamma \alpha} ~\alpha,\beta,\gamma =1 \cdots 3.
\label{eq:13a}
\end{eqnarray}Here $b, c>0$ are material--dependent bulk
constants, independent of the temperature, whereas the parameter
$a$ scales linearly with the absolute temperature and is given by
\begin{equation}
a = \alpha \left( T - T^*\right) \label{eq:a}
\end{equation}
where $\alpha>0$ and $T^*$ is a characteristic liquid crystal
temperature \cite{dg, newtonmottram}.

The equilibrium, physically observable configurations correspond
to global or local minimizers of the Landau-De Gennes energy
functional, $I_{LG}$, subject to the imposed boundary conditions.
In what follows, we first consider spatially homogeneous cases in
Section~\ref{sec:bulk} and then study global energy minimizers in
spatially inhomogeneous cases in Sections~\ref{sec:oc} and
\ref{sec:max}.

\subsection{The Bulk Energy Density}
\label{sec:bulk}

Our first proposition concerns the stationary points of the bulk
energy density. Proposition~\ref{prop:uniaxial} is known in the
literature \cite{forest2} and we give an alternative proof here
for completeness.

\begin{prop}\cite{ball}
The stationary points of the bulk energy density, $f_B$ in
(\ref{eq:13}), are given by either uniaxial or isotropic
$\Qvec$-tensors of the form
\begin{equation}
\Qvec = d\left(\nvec\otimes\nvec - \frac{1}{3}\mathbf{I}\right)
\label{eq:uniaxial}
\end{equation} where $d$ is a scalar order parameter and $\nvec$
is one of the eigenvectors of $\Qvec$ in (\ref{eq:6}).
On comparing (\ref{eq:uniaxial}) with (\ref{eq:8}), we see that
when $\nvec=\evec_1$, the parameter $d=s$ and the scalar order
parameter $r=0$. Similarly, when $\nvec=\evec_2$, the parameter
$d=r$ and $s=0$ whereas when $\nvec=\evec_3$, the order parameters
$s,r$ in (\ref{eq:8}) are equal and are given by $s=r = -d$. \label{prop:uniaxial}
\end{prop}

\begin{proof}For a symmetric, traceless matrix $\Qvec$ of the
form (\ref{eq:6}), $\textrm{tr}\Qvec^n = \sum_{i=1}^{3}\la_i^n$
subject to the tracelessness condition so that the bulk energy
density $f_B$ only depends on the eigenvalues $\lambda_1,
\lambda_2$ and $\lambda_3$. Then the stationary points of $f_B$
are given by the stationary points of the function $f:\Rr^3
\rightarrow \Rr$ defined by
\begin{equation}
f\left(\la_1,\la_2,\la_3\right) = \frac{a}{2}\sum_{i=1}^{3}\la_i^2
- \frac{b}{3}\sum_{i=1}^{3}\la_i^3 +
\frac{c}{4}\left(\sum_{i=1}^{3}\la_i^2 \right)^2 -
2\delta\sum_{i=1}^{3}\la_i \label{eq:16}
\end{equation} where we have recast $f_B$ in terms of the eigenvalues
and introduced a Lagrange multiplier $\delta$ for the tracelessness condition.

The equilibrium equations are given by a system of three algebraic
equations
\begin{equation}
\frac{\partial f}{\partial \lambda_i} = 0 \Leftrightarrow a \la_i
- b\la_i^2 + c\left(\sum_{k=1}^{3}\la_k^2 \right)\la_i = 2\delta,
\quad \textrm{for $i=1 \ldots 3$,} \label{eq:17}
\end{equation} together with the tracelessness condition $\sum_i \la_i = 0$. The system (\ref{eq:17}) is equivalent to
\begin{equation}
\left( \la_i - \la_j \right)\left[ a - b\left(\la_i + \la_j
\right) + c \sum_{k=1}^{3}\la_k^2 \right] = 0 \quad 1\leq i< j
\leq 3.\label{eq:18}
\end{equation}
Let $\left\{\lambda_i\right\}$ be a solution of the system
(\ref{eq:17}) with three distinct eigenvalues $\la_1 \neq \la_2
\neq \la_3$.  We consider equation (\ref{eq:18}) for the pairs
$\left(\la_1,\la_2\right)$ and $\left(\la_1,\la_3\right)$. This
yields two equations
\begin{eqnarray}
&& a - b\left(\la_1 + \la_2\right) + c\sum_{k=1}^{3}\la_k^2 = 0
\nonumber \\&& a - b\left(\la_1 + \la_3\right) +
c\sum_{k=1}^{3}\la_k^2 = 0 \label{eq:19}
\end{eqnarray}
from which we obtain
\begin{equation}
-b\left(\la_2 - \la_3 \right) = 0, \label{eq:20}
\end{equation} contradicting our initial hypothesis $\la_2 \neq
\la_3$. We, thus, conclude that a stationary point of $f_B$ must
have at least two equal eigenvalues and therefore correspond to
either an uniaxial or isotropic liquid crystal state. In
particular, there are no biaxial stationary points for the
particular choice of $f_B$ in (\ref{eq:13}).
\end{proof}

By virtue of Proposition~\ref{prop:uniaxial}, it suffices to
consider uniaxial $\Qvec$-tensors of the form $$\Qvec =
s\left(\nvec\otimes \nvec - \frac{1}{3}\mathbf{I}\right) \quad
\nvec\in S^2$$ whilst computing the stationary points of $f_B$.
For such $\Qvec$-tensors, $f_B$ is a quartic polynomial in the
uniaxial scalar order parameter $s$ and the stationary points are
the roots of the algebraic equation given
below
\begin{equation}
\frac{d f_B}{d s} = \frac{1}{27}\left(18as - 6bs^2 + 12cs^3\right)
= 0. \label{eq:f2}
\end{equation}
There are precisely three stationary points;
\begin{equation}
s = 0 \quad \textrm{and $s_{\pm} = \frac{b \pm \sqrt{b^2 - 24
ac}}{4c}$} \label{eq:f3}
\end{equation}
where
\begin{equation}
f_B(0) = 0 \quad \textrm{and} \quad f_B(s_\pm) =
\frac{s_{\pm}^2}{54}\left(9a - b s_{\pm} \right), \label{eq:f4}
\end{equation} and $f_B(s_-) > f_B(s_+)$.
Hence, the global bulk energy minimizer is either the isotropic
state $\Qvec=0$ or the ordered nematic state
 \begin{equation}
\Qvec = s_+\left(\evec\otimes\evec - \frac{1}{3}\mathbf{I}\right)
 \label{eq:uniaxial2}
\end{equation} where $\evec$ is the eigenvector with the largest
eigenvalue.

A natural question is - for which temperature ranges does the
global bulk energy minimizer lie inside the physical triangle
$T_\psi$ i.e. for which temperature regimes does $s_+$, which is
the stable nematic stationary point, take values in the physical
range $$ 0 \leq s_+= \frac{b + \sqrt{b^2 - 24ac}}{4c} \leq 1 ? $$
One can directly verify that $s_+\in\left[0,1\right]$ if and only
if
\begin{equation}
\label{eq:b1} \frac{1}{3}\left(b- 2c\right) \leq a \leq
\frac{b^2}{24c},
\end{equation}
or equivalently, in terms of the absolute temperature $T$ if and
only if
\begin{equation}
\label{eq:b1new} \frac{1}{3\alpha}\left(b- 2 c\right)+T^* \leq T
\leq \frac{b^2}{24\alpha c}+ T^*.
\end{equation} For the common liquid crystal material MBBA, the
values of the characteristic bulk constants are given in the
literature \cite{gartland2,newtonmottram}
\begin{eqnarray}
\label{eq:b5}
&& \alpha = 0.42\times 10^3 J/m^3~^{o}C, ~b = 0.64\times 10^4 J/m^3,~c = 0.35\times 10^4 J/m^3 \nonumber\\
&& T^* = 45^{o}C \quad T_c = 46^{o}C\end{eqnarray} where $T_c$ is
the nematic-isotropic transition temperature. We substitute these
values into (\ref{eq:b1}) and (\ref{eq:b1new}) and find that $s_+
>1$ for $T < 44.52^{o}C$ i.e. $s_+$ moves outside the physical
range within a $2^{o}C$ - neighbourhood of the nematic-isotropic
transition temperature.

We recall that there are three characteristic temperatures
predicted by the quartic form of $f_B$ in (\ref{eq:13}): (i)
$a=0$, below which the isotropic state loses its stability (ii)
the nematic-isotropic transition temperature, $a=\alpha(T_c - T^*)
= \frac{b^2}{27c}$, for which $f_B(s_+) = f_B(0)$ and (iii) $a =
\frac{b^2}{24c}$ above which the ordered nematic stationary points
are no longer defined in (\ref{eq:f3}). We provide a pictoral
representation for the stationary points of $f_B$
for ease of comparison with $T_\psi$. We define the \emph{bulk
triangle}, $\triangle(T)$, to be the convex hull of the stationary
points of $f_B$ in the order-parameter $(s,~r)$ - plane. For
$-\alpha T^* \leq a < -\frac{b^2}{3c}$, $\triangle(T)$ is an
isosceles triangle with its vertices at the points
$\left\{(2|s_-|,0),~(0,2|s_-|),~(-2|s_-|,-2|s_-|)\right\}$ whereas
for $-\frac{b^2}{3c} \leq a \leq \frac{b^2}{24c}$, $\triangle(T)$
is an isosceles triangle with its vertices at the points
$\left\{(s_+,0),~(0,s_+),~(-s_+,-s_+)\right\}$. For $a >
\frac{b^2}{24c}$, $\triangle(T)$ collapses to the origin since
$s=0$ is the unique critical point in the high-temperature regime.
In Figures~\ref{fig:2} and \ref{fig:3}, we illustrate
$\triangle(T)$ for all temperature regimes.

\begin{figure}
[hp]
\begin{center}
\includegraphics[width=5 in, height=2.8 in]{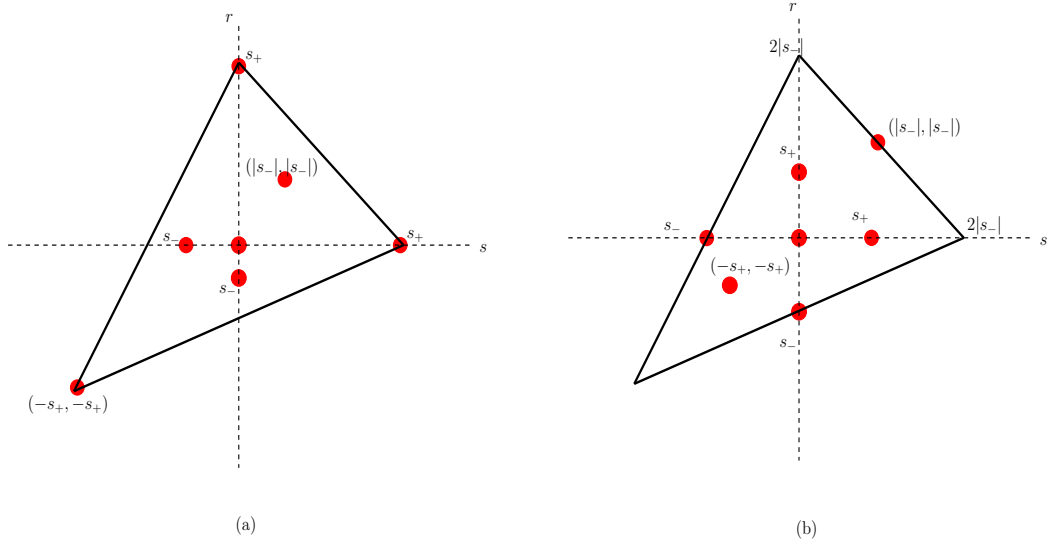}
\caption{(a) The triangle $\triangle(T)$ for
$-\frac{b^2}{3c\alpha}+T^* \leq T < T^*$. The red marked points
label the stationary points of $f_B$ in this temperature regime. (b) The
triangle $\triangle(T)$ for $T <
-\frac{b^2}{3c\alpha}+T^*$.}\label{fig:2}
\end{center}
\end{figure}

\begin{figure}
[hp]
\begin{center}
\includegraphics[width=3 in, height=3 in]{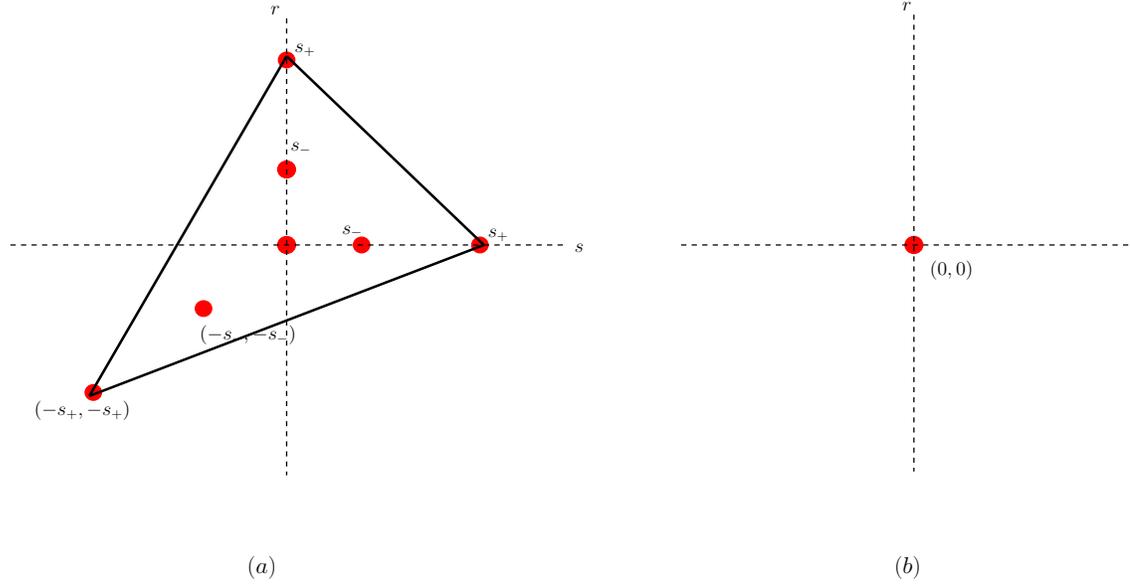}
\caption{The triangle $\triangle(T)$ for $T^*\leq T <
\frac{b^2}{24\alpha c}+ T^*$. The stationary point $s_{-} >0$ in
this temperature regime. (b) The triangle $\triangle(T)$ shrinks
to the origin for $T>\frac{b^2}{24\alpha c}+T^*$, since the
isotropic state $(s,~r)=(0,~0)$ is the unique stationary point in
this temperature regime.}\label{fig:3}
\end{center}
\end{figure}

\subsection{The One--constant Elastic Energy Density}
\label{sec:oc}

In this section, we study global minimizers of the Landau-De
Gennes energy functional, $I_{LG}$ in (\ref{eq:12}),
$$I_{LG}\left[\Qvec\right] = \int_{\Omega} f_B\left(\Qvec\right) +
L|\grad \Qvec|^2~dV$$ and obtain explicit bounds for the
equilibrium scalar order parameters, $(s,~r)$ in (\ref{eq:9}).
These bounds quantify the effect of the elastic energy density on
the bulk energy minima and can also be compared to the
probabilistic bounds in (\ref{eq:7}).

\subsubsection{Existence and regularity of minimizers}
There exists a global minimizer, $\Qvec^*$, of $I_{LG}$ in the
admissible class $\mathcal{A}$, where $\Acal$ has been defined in
(\ref{eq:31}). This is a ready consequence of the direct methods
in the calculus of variations \cite{dacorogna, gartland}. Indeed,
one can check that $I_{LG}$ satisfies the following coerciveness
estimate \cite{gartland}
\begin{equation}
I_{LG}[\Qvec] \geq \alpha_c ||\Qvec||_{W^{1,2}(\Omega)}^2
\label{eq:34}
\end{equation} where $\alpha_c>0$ and the $W^{1,2}$-norm,
$||\Qvec||_{W^{1,2}(\Omega)}$, is given by
$$||\Qvec||_{W^{1,2}(\Omega)} = \left(\int_{\Omega} |\Qvec|^2 + |\grad \Qvec|^2~dV \right)^{1/2}.$$
The Landau-De Gennes energy density is convex in the gradient
$\grad\Qvec$ and therefore $I_{LG}$ is weakly lower semicontinuous
\cite{evans}. The coerciveness and weak lower semicontinuity of
$I_{LG}$ guarantee that the infimum energy is actually achieved
i.e. there exists a $\Qvec^*\in \mathcal{A}$ with the property
\begin{equation}
I_{LG}[\Qvec^*] = \inf_{\Qvec\in\mathcal{A}}
I_{LG}[\Qvec].\label{eq:35}
\end{equation}

The global minimizer $\Qvec^*$ is a weak solution of the
corresponding Euler-Lagrange equations, which is a system of
nonlinear elliptic partial differential equations as shown below -
\begin{equation}
2 L\Delta \Qvec_{\alpha\beta} = a \Qvec_{\alpha\beta} +
b\left(\frac{1}{3}\textrm{tr}\Qvec^2 \delta_{\alpha\beta} -
\Qvec_{\alpha p}\Qvec_{p \beta}\right) +
c\left(\Qvec_{pq}\right)^2\Qvec_{\alpha\beta} \qquad \alpha,\beta
= 1 \ldots 3. \label{eq:36}
\end{equation} We use standard results from the theory of elliptic
partial differential equations to deduce that $\Qvec^*$ is
actually a classical solution of the system (\ref{eq:36}) and
$\Qvec^*$ is smooth and analytic everywhere in $\Omega$
\cite{gilbarg}. Given smooth boundary conditions,
$\Qvec^*$ is also smooth up to the boundary.

\subsubsection{Upper bounds for the order parameters}

The global minimizer $\Qvec^*$ can be expressed in terms of a pair
of eigenvectors $(\nvec^*,\mvec^*)$ and the scalar order
parameters $(s^*,~r^*)$, as in (\ref{eq:8}),
\begin{equation}
\Qvec^* = s^*\left(\nvec^*\otimes\nvec^* -
\frac{1}{3}\mathbf{I}\right) + r^*\left(\mvec^*\otimes\mvec^* -
\frac{1}{3}\mathbf{I}\right) \label{eq:new}
\end{equation}
and
\begin{equation}
|\Qvec^*|^2 = \frac{2}{3}\left(s{^{*}}^2 + r{^*}^2 - s^*r^*
\right). \label{eq:38}
\end{equation}
We partition the $(s,~r)$-plane into three regions: (a) $R_1 =
\left\{(s,~r)|s,r \geq 0\right\}$ - the top quadrant, (b) $R_2 =
\left\{(s,~r)| s \leq 0; ~ r \geq s \right\}$ and (c) $R_3 =
\left\{ r\leq 0; ~ r\leq s \right\}$. In $R_1$, we have the
inequalities
\begin{equation}
\frac{1}{6}\left(s^* + r^* \right)^2 \leq |\Qvec^*|^2 \leq
\frac{2}{3}\left(s^* + r^* \right)^2. \label{eq:39}
\end{equation} Similarly, for $R_2$, we have that
\begin{equation}
\frac{1}{6}\left(r^* - 2s^*\right)^2 \leq |\Qvec^*|^2 \leq
\frac{2}{3}\left(r^* - 2s^* \right)^2 \label{eq:40}
\end{equation} and for $R_3$,
\begin{equation}
\frac{1}{6}\left(s^* - 2r^*\right)^2 \leq |\Qvec^*|^2 \leq
\frac{2}{3}\left(s^* - 2r^*\right)^2. \label{eq:41}
\end{equation}

For every $\eta>0$, we define the bounded region $\Sigma_{\eta} =
\left\{ (s^*,r^*)~;~\left|\Qvec^*\right| \leq \eta \right\}$ in
the $(s,r)$-plane. Let $T_{\eta}$ be the isosceles triangle in the
$(s,r)$-plane with its vertices at the points -
$\left\{\left(\eta, 0\right), \left(0,\eta\right), \left(-\eta,
-\eta \right)\right\}$. Then it follows immediately from
(\ref{eq:39}), (\ref{eq:40}) and (\ref{eq:41}) that
\begin{equation}
T_{\sqrt{\frac{3}{2}}~\eta} \subset \Sigma_{\eta}\subset
T_{\sqrt{6}~\eta} \label{eq:region}
\end{equation} so that $T_{\sqrt{6}\eta} \subseteq T_{\psi}$ necessarily implies that
$\Sigma_{\eta} \subset T_\psi$. Our first result in this section
is an explicit upper bound for the norm of a global energy
minimizer in the low-temperature regime $a \leq \frac{b^2}{24c}$
and as the preceding discussion shows, this upper bound allows us
to define the admissible domain for the equilibrium scalar order
parameters.


\begin{thm} \label{thm:1}
Let $\Qvec^*\in \mathcal{A}$ be a global minimizer for the energy
functional $I_{LG}$, where $\mathcal{A}$ and $I_{LG}$ have been
defined in (\ref{eq:31}) and (\ref{eq:12}) respectively. We work
in the temperature regime $a\leq \frac{b^2}{24c}$ and make the
following assumption about the boundary condition $\Qvec_0$,
\begin{equation}
|\Qvec_0(\rvec)|< \min\left\{\frac{b + \sqrt{b^2 -
24ac}}{4\sqrt{6}c},~\frac{1}{\sqrt{6}}\right\} \quad \rvec\in
\partial\Omega, \label{eq:42}
\end{equation} where $|\Qvec|$ has been defined in (\ref{eq:matrixnorm}).
The condition (\ref{eq:42}) is equivalent to requiring that the
boundary order parameters are contained inside $T_{\psi}$ and the
bulk triangles $\triangle(T)$ defined in Section~3.1. Then
$\Qvec^*$ obeys the following global upper bound on
$\overline{\Omega}$ -
\begin{equation}
|\Qvec^*(\rvec)| \leq \frac{b + \sqrt{b^2 - 24ac}}{2\sqrt{6}c}
~~~\textrm{for $\rvec \in \overline{\Omega}$.} \label{eq:43}
\end{equation}
\end{thm}

\begin{proof} We assume that the contrary holds
i.e. the subset $\Omega^* = \left\{\rvec\in \Omega;~
|\Qvec^*(\rvec)| > \frac{b + \sqrt{b^2 -
24ac}}{2\sqrt{6}c}\right\}\subset \Omega$ has positive measure.
The subset $\Omega^*$ clearly does not intersect $\partial\Omega$
since the boundary condition $\Qvec_0$ obeys this upper bound from
assumption (\ref{eq:42}).

We define a perturbation, $\widetilde{\Qvec}$, of the global
minimizer $\Qvec^*$ as follows --
\begin{equation}
  \label{eq:44}
  \widetilde{\Qvec}(\rvec) =
  \begin{cases}
    \Qvec^*(\rvec), ~ \rvec \in \Omega\setminus\Omega^*,\\
      \\
 \frac{\Gamma}{|\Qvec^*(\rvec)|} \Qvec^*(\rvec), ~ \rvec \in
 \Omega^*
  \end{cases}
\end{equation} where
\begin{equation}
\Gamma = \frac{b + \sqrt{b^2 - 24ac}}{2\sqrt{6}c}.
\label{eq:alpha}
\end{equation} We note from (\ref{eq:f3}) that
\begin{equation}
\label{eq:alpha2} \Gamma = \sqrt{\frac{2}{3}}s_+ \end{equation}by
definition. It is evident from (\ref{eq:44}) that
$\widetilde{\Qvec}$ agrees with $\Qvec^*$ everywhere outside
$\Omega^*$ and hence belongs to our admissible space. Moreover,
$\widetilde{\Qvec}$ has constant norm on the set $\Omega^*$ i.e.
$|\widetilde{\Qvec}(\rvec)| = \Gamma$ for $\rvec\in \Omega^*$.

 We obtain an upper bound for the free energy difference
 \begin{equation}
 I_{LG}[\widetilde{\Qvec}] - I_{LG}[\Qvec^*] =
 \int_{\Omega^*} f_B(\widetilde \Qvec) + L|\grad\widetilde\Qvec|^2 -
\left(f_B\left(\Qvec^*\right)+ L|\grad\Qvec^*|^2\right) ~ dV
\label{eq:45}
\end{equation} where $f_B$ is as in (\ref{eq:13}).

We can explicitly compute $\left|\grad\widetilde{\Qvec}\right|^2$
as shown below ,
\begin{equation}
\left|\grad\widetilde{\Qvec}(\rvec)\right|^2 =
\left(\frac{\Gamma}{|\Qvec^*(\rvec)|}\right)^2\left(|\grad\Qvec^*|^2
-
\frac{1}{|\Qvec^*|^2}\left(\Qvec^*_{pq}\Qvec^*_{pq,k}\right)\left(\Qvec^*_{ij,k}\Qvec^*_{ij,k}\right)\right)
\leq |\grad \Qvec^*(\rvec)|^2 \label{eq:46}
\end{equation}
since $ \left(\frac{\Gamma}{|\Qvec^*(\rvec)|}\right)^2 < 1 $ on
$\Omega^* $  by definition.

Consider the function $G:[0,\infty)\rightarrow \Rr$ defined by
\begin{equation}
G(u) = - u^2\left(\frac{a}{2} - \frac{b}{3\sqrt{6}}u +
\frac{c}{4}u^2 \right). \label{eq:48}
\end{equation}
We estimate the bulk energy density difference in terms of the
function $G$ as follows
\begin{eqnarray}
&& f_B(\widetilde\Qvec) - f_B\left(\Qvec^*\right) =
\frac{a}{2}\textrm{tr}\widetilde{\Qvec}^2 -
\frac{b}{3}\textrm{tr}\widetilde{\Qvec}^3 +
\frac{c}{4}\left(\textrm{tr}\widetilde{\Qvec}^2\right)^2 -
\left(\frac{a}{2}\textrm{tr}\Qvec{^*}^2 -
\frac{b}{3}\textrm{tr}\Qvec{^*}^3 +
\frac{c}{4}\left(\textrm{tr}\Qvec{^*}^2\right)^2 \right) = \nonumber \\
&& = \frac{a}{2}\left(\Gamma^2 - |\Qvec^*|^2\right) -
\frac{b}{3}\frac{\textrm{tr}\Qvec{^*}^3}{|\Qvec^*|^3}\left(\Gamma^3
- |\Qvec^*|^3 \right) + \frac{c}{4}\left(\Gamma^4 -
|\Qvec^*|^4\right) \leq G\left(|\Qvec^*|\right) -
G\left(\Gamma\right) \label{eq:47}.
\end{eqnarray} In the last step of (\ref{eq:47}), we use the
equality
$$\textrm{tr}\widetilde{\Qvec}^3 =\Gamma^3 \frac{\textrm{tr}\Qvec{^*}^3}{|\Qvec^*|^3},$$
 $\left(\Gamma^3 - |\Qvec^*|^3 \right)< 0$ on $\Omega^*$ and
the inequality
\begin{equation}
\label{eq:biaxiality1} \frac{\textrm{tr}\Qvec{^*}^3}{|\Qvec^*|^3}
\leq \frac{1}{\sqrt{6}},
\end{equation}
from Lemma~\ref{lem:1}. One can readily verify that $G(u)$ attains
a local maximum for $u=\Gamma$ and $G'(u) < 0$ for all $u >
\Gamma$. Therefore,
\begin{equation}
G\left(|\Qvec^*|\right) - G(\Gamma) < 0,  \label{eq:50}
\end{equation}
since $|\Qvec^*| > \Gamma $ on $\Omega^*$ by definition.

We substitute (\ref{eq:46}), (\ref{eq:47}) and (\ref{eq:50}) into
(\ref{eq:45}) to obtain
\begin{equation}
I_{LG}\left[\widetilde{\Qvec}\right] - I_{LG}[\Qvec^*] < 0,
\label{eq:51}
\end{equation}
contradicting the absolute energy minimality of $\Qvec^*$. We thus
conclude that $\Omega^*$ is empty and
\begin{equation}
|\Qvec^*(\rvec)| \leq \Gamma = \frac{b + \sqrt{b^2 -
24ac}}{2\sqrt{6}c} \label{eq:52}
\end{equation}
for all points $\rvec \in \overline{\Omega}$.
\end{proof}

\begin{lem}
\label{lem:1} Let $\beta(\Qvec)$ be defined as follows -
\begin{equation}
\label{eq:biaxiality2} \beta(\Qvec) = 1 -
6\frac{\left(\textrm{tr}\Qvec^3\right)^2}{\left(\textrm{tr}\Qvec^2\right)^3}
\quad \Qvec \in S_0.
\end{equation} Then $0\leq \beta(\Qvec)\leq 1$.
\end{lem}
\begin{proof} The quantity $\beta(\Qvec)$ is known as the
biaxiality parameter in the liquid crystal literature and it is
well-known that $\beta(\Qvec)\in\left[0,1\right]$
\cite{gartland2}. We present a simple proof here for completeness.

Since
$6\frac{\left(\textrm{tr}\Qvec^3\right)^2}{\left(\textrm{tr}\Qvec^2\right)^3}
\geq 0$, the inequality $\beta(\Qvec) \leq 1$ is trivial. To show
$\beta(\Qvec)\geq 0$, we use the representation (\ref{eq:8}) to
express $\textrm{tr}\Qvec^3$ and $\textrm{tr}\Qvec^2$ in terms of
the order parameters $s$ and $r$.
\begin{eqnarray}
\label{eq:biaxiality3} && \textrm{tr}\Qvec^3 =
\frac{1}{9}\left(2s^3 + 2r^3 - 3s^2 r - 3sr^2\right)\nonumber \\
&& \textrm{tr}\Qvec^2 = \frac{2}{3}\left( s^2 + r^2 - sr \right)
\end{eqnarray} A straightforward calculation shows that
$$\left(\textrm{tr}\Qvec^3\right)^2 = \frac{1}{81}\left(4s^6 +
4r^6 - 12 s^5 r - 12 s r^5 + 26 s^3 r^3 - 3 s^4 r^2 - 3s^2 r^4
\right)$$ and
$$\left(\textrm{tr}\Qvec^2\right)^3 = \frac{8}{27}\left(s^6 + r^6
- 3 s^5 r - 3 s r^5 - 7 s^3 r^3 + 6 s^2 r^4 + 6 s^4 r^2 \right).$$
One can then directly verify that
\begin{equation}
\label{eq:biaxiality4}\left(\textrm{tr}\Qvec^2\right)^3 -
6\left(\textrm{tr}\Qvec^3\right)^2 = 2 s^2 r^2\left( s- r
\right)^2 \geq 0
\end{equation}
as required.
\end{proof}

As a further illustration, let us assume that there exists an
uniaxial global energy minimizer in the admissible space, $\Acal$
in (\ref{eq:31}), where the boundary condition $\Qvec_0$ is of the
form
\begin{equation}
\label{eq:unbc}\Qvec_0 = s_0 \left(\nvec_0 \otimes \nvec_0 -
\frac{1}{3}\mathbf{I}\right),~\nvec_0: \partial \Omega \to S^2
\end{equation}
and $0 < s_0 < \min\left\{s_+, 1\right\}$ is a positive constant.
Then we have
\begin{lem}
\label{lem:2} Let $\Qvec_u$ be an uniaxial global minimizer of
$I_{LG}$ in the admissible space $\Acal$, with a smooth, uniaxial
and physically realistic boundary condition $\Qvec_0$, as in
(\ref{eq:unbc}). Then $\Qvec_u$ is necessarily of the form
\begin{equation}
\label{eq:uniglobalminimizer} \Qvec_u = s_u\left(\nvec_u\otimes
\nvec_u - \frac{1}{3}\mathbf{I}\right)
\end{equation} for some function $s_u: \bar{\Omega}\to \Rr$ and
unit-vector field $\nvec_u:\bar{\Omega} \to S^2$. The equilibrium
scalar order parameter is non-negative everywhere and obeys the
following inequalities
\begin{equation}
\label{eq:unscalar} 0\leq s_u(\rvec) \leq s_+ \quad \rvec \in
\bar{\Omega}.
\end{equation}
\end{lem}

\begin{proof}
We prove Lemma~\ref{lem:2} by contradiction. Let $\Omega^* =
\left\{\rvec\in\Omega;~s_u(\rvec) < 0 \right\}$ be a measurable
interior subset of $\Omega$. The subset $\Omega^*$ does not
intersect $\partial \Omega$ by virtue of our choice of $\Qvec_0$
in (\ref{eq:unbc}). Consider the perturbation
\begin{equation}
  \label{eq:uniglobalminimizer2}
  \widetilde{\Qvec}_u(\rvec) =
  \begin{cases}
    \Qvec_u(\rvec), ~ \rvec \in \Omega\setminus\Omega^*,\\
      \\
 -\Qvec_u, ~ \rvec \in
 \Omega^*.
  \end{cases}
\end{equation}Then $\widetilde{\Qvec}_u \in \Acal$ and
$\widetilde{\Qvec}_u$ coincides with $\Qvec_u$ everywhere outside
$\Omega^*$. We explicitly estimate the free energy difference,
$I_{LG}[\widetilde{\Qvec}_u] - I_{LG}[\Qvec_u]$, as shown below
\begin{eqnarray}
&& I_{LG}[\widetilde{\Qvec}_u] - I_{LG}[\Qvec_u] =
\int_{\Omega^*}f_B(\widetilde{\Qvec}_u) +
L|\grad\widetilde{\Qvec}_u|^2 - \left(f_B\left(\Qvec_u\right)+
L|\grad\Qvec_u|^2\right) ~ dV =  \nonumber \\ && = \int_{\Omega^*}
\frac{2b}{3}\textrm{tr}\Qvec_u^3~dV =
\int_{\Omega^*}\frac{4b}{27}s_u^3~dV < 0
\label{eq:uniglobalminimizer3}
\end{eqnarray}
since $\textrm{tr}\Qvec_u^3 = \frac{2}{9}s_u^3$, $s_u < 0$ on
$\Omega^*$ by assumption and $b>0$. This contradicts the absolute
energy minimality of $\Qvec_u$. Hence, $\Omega^*$ is empty and
$s_u \geq 0$ everywhere in $\Omega$.

The upper bound in (\ref{eq:unscalar}) follows directly from
(\ref{eq:43}) i.e.
\begin{equation}
\label{eq:uniglobalminimizer4} |\Qvec_u| = \sqrt{\frac{2}{3}}|s_u|
\leq \sqrt{\frac{2}{3}} s_+
\end{equation}
and since $s_u \geq 0$, we have that $0\leq s_u \leq s_+$ from
(\ref{eq:uniglobalminimizer4}). Lemma~\ref{lem:2} now follows.
\end{proof}

 \textbf{Remark:} We are not guaranteed the existence of a
 uniaxial global energy minimizer in our admissible space.
 We assume the existence of $\Qvec_u$ in Lemma~\ref{lem:2}.
 A more
 technically precise formulation of the problem would be to
 minimize $I_{LG}$ in the restricted class of uniaxial
 $\Qvec$-tensors
 $$\Acal_u = \left\{\Qvec\in W^{1,2}\left(\Omega;S_0\right);~
 \Qvec= s\left(\nvec\otimes \nvec - \frac{1}{3}\mathbf{I}\right)
 a.e. ~in~\Omega\right\} $$ where $s$ is a real-valued function
 and $\nvec$ is a unit-vector field, subject to uniaxial boundary
 conditions. In this case, one can prove the existence of a
 uniaxial global minimizer $\Qvec_u$ in the restricted class
 $\Acal_u$ and the statement of Lemma~\ref{lem:2} still
 holds. However, the proof is technically more involved and we omit
 the details for brevity.

Given the explicit upper bound $\Gamma$ in (\ref{eq:43}) for the
norm of a global energy minimizer, the corresponding equilibrium
scalar order parameters are confined to the bounded region
$\Sigma_\Gamma = \left\{(s,~r);~ |\Qvec^*| \leq \Gamma \right\}$
in the $(s,~r)$-plane. We define the elastic triangle
$\triangle_{el}(T)$ to be the triangle $T_{\sqrt{6}\Gamma}$, where
$T_{\sqrt{\frac{3}{2}\Gamma}}\subset \Sigma_\Gamma \subset
T_{\sqrt{6}\Gamma}$ in (\ref{eq:region}). The elastic triangle can
be explicitly specified in terms of the temperature and the
material-dependent bulk constants. One can directly verify that
$\Sigma_\Gamma \subset \triangle_{el}(T)\subseteq T_{\psi}$ if and
only if
\begin{equation}
\label{eq:e11new} \frac{b-c}{6\alpha} + T^* \leq T \leq
\frac{b^2}{24 \alpha c} + T^*. \end{equation} and $\Sigma_\Gamma
\supset T_{\sqrt{\frac{3}{2}}\Gamma} \supseteq T_{\psi}$ if and
only if
\begin{equation}
\label{eq:el2} T \leq \frac{1}{3\alpha}\left(b - 2c\right)+T^*.
\end{equation} In other words, for temperatures $T \in\left[0,\frac{1}{3\alpha}\left(b -
2c\right)+T^*\right)$, the equilibrium order parameters may move
outside the physical triangle. For the liquid crystal material
MBBA, $\Sigma_\Gamma \supset T_\psi$ for $T < 44.52^{o}C$ i.e. the
Landau-De Gennes predictions fail to be consistent with the
probabilistic second-moment definition of $\Qvec$ within a
$2^{o}C$-neighbourhood of the nematic-isotropic transition
temperature.

\subsection{Maximum principle approach}
\label{sec:max}

In this section, we carry out a parallel analysis in the
high-temperature regime $a>\frac{b^2}{24c}$ and extend our
analysis to more general Landau-De Gennes energy functionals,
where $f_B$ is a general even polynomial in the $\Qvec$-tensor
components.

\begin{thm} Let $\Qvec^*$ be a global minimizer of
$I_{LG}$ in the admissible class $\mathcal{A}$, where $I_{LG}$ and
$\mathcal{A}$ have been defined in (\ref{eq:12}) and (\ref{eq:31})
respectively, in the temperature regime $a> \frac{b^2}{24c}$. Then
the function $|\Qvec^*|:\overline{\Omega}\rightarrow \mathbb{R} $
attains its maximum on the domain boundary. In particular, if the
boundary condition $\Qvec_0$ satisfies
\begin{equation} \label{eq:bc} |\Qvec_0(\rvec)| <
\frac{1}{\sqrt{6}} \quad \rvec\in
\partial\Omega,
\end{equation} then the scalar
order parameters of $\Qvec^*$ are contained inside $T_{\psi}$, for
this high-temperature regime. \label{thm:2}
\end{thm}

\begin{proof} The proof proceeds by contradiction. We work in the
temperature regime $a> \frac{b^2}{24c}$ and assume that the
function $|\Qvec^*|:\overline{\Omega}\rightarrow \mathbb{R}$
attains a strict maximum at an interior point $\rvec^*\in\Omega$,
where $|\Qvec^*(\rvec^*)|>0$.

The global minimizer $\Qvec^*$ is a classical smooth solution of
the Euler-Lagrange equations
\begin{equation}
2 L\Delta \Qvec_{\alpha\beta} = a \Qvec_{\alpha\beta} +
b\left(\frac{1}{3}\textrm{tr}\Qvec^2 \delta_{\alpha\beta} -
\Qvec_{\alpha p}\Qvec_{p \beta}\right) +
c\left(\Qvec_{pq}\right)^2\Qvec_{\alpha\beta} \quad
L>0,~\alpha,\beta = 1 \ldots 3. \label{eq:EL}
\end{equation}
Therefore, the function $|\Qvec^*|^2:\overline{\Omega}\rightarrow
\mathbb{R}$ is also a smooth function and we necessarily have that
\begin{equation}\Delta |\Qvec^*|^2  \leq 0 \quad \textrm{at
$\rvec^*\in\Omega$},
\label{eq:max2}
\end{equation} according to our hypothesis \cite{evans}.
We compute $\Delta |\Qvec^*|^2$ at this interior
maximum point. One can readily show that
\begin{equation}
L\Delta |\Qvec^*|^2 = 2 L \left(|\grad\Qvec^*|^2 +
\Qvec^*_{ij}\Qvec^*_{ij,kk}\right). \label{eq:max3}
\end{equation}
We substitute the Euler-Lagrange equations (\ref{eq:EL}) into
(\ref{eq:max3}) to obtain the following
\begin{eqnarray}
&& L\Delta |\Qvec^*|^2 = 2L |\grad\Qvec^*|^2 + \left(a|\Qvec^*|^2 - b\textrm{tr}\left(\Qvec^*\right)^3 + c|\Qvec^*|^4\right) + \frac{b}{3}\left(\textrm{tr}\Qvec{^*}^2\right)\Qvec^*_{ii}\nonumber \\
&& \qquad = 2 L |\grad\Qvec^*|^2 + a\textrm{tr}\Qvec{^*}^2 -
b\textrm{tr}\Qvec{^*}^3 + c\left(\textrm{tr}\Qvec{^*}^2\right)^2,
\label{eq:max4}
\end{eqnarray} since $\Qvec^*_{ii} = 0$.

Consider the function $G:S_0 \rightarrow \mathbb{R}$ defined by
\begin{equation}
G(\Qvec) = a\textrm{tr}\Qvec^2 - b\textrm{tr}\Qvec^3 +
c\left(\textrm{tr}\Qvec^2\right)^2. \label{eq:max5}
\end{equation} Then $G$ is bounded from below by
$$ G(\Qvec) \geq h(|\Qvec|) = a |\Qvec|^2 - \frac{b}{\sqrt{6}}|\Qvec|^3 + c|\Qvec|^4 $$
from Lemma~\ref{lem:1}. The function $h:\bar{\Omega} \to \Rr$ has
its global minimum at the isotropic state, $\Qvec=0$, and
$$h(|\Qvec|) > 0 \quad \Qvec \neq 0$$
in the temperature regime $a>\frac{b^2}{24c}$. This implies that
$$ G(\Qvec^*) > 0 \quad \textrm{at $\rvec^*\in\Omega$} $$
and consequently $\Delta |\Qvec^*|^2(\rvec^*) > 0$ from
(\ref{eq:max4}). This contradicts our hypothesis and
Theorem~\ref{thm:2} now follows.

In particular, if the boundary condition $\Qvec_0$ satisfies the
hypothesis (\ref{eq:bc}), then the global energy minimizer
$\Qvec^*$ satisfies the inequality
\begin{equation}
\label{eq:max9} |\Qvec^*(\rvec)|< \frac{1}{\sqrt{6}}
\end{equation} on $\overline{\Omega}$.
From (\ref{eq:region}), this is sufficient to ensure that
$\Qvec^*$ is physically realistic, in the sense that its scalar
order parameters do not take values outside $T_{\psi}$.
\end{proof}

Theorem~\ref{thm:2} shows that in the high temperature regime
$a>\frac{b^2}{24c}$, the norm of a global energy minimizer attains
its maximum on the boundary. The isotropic state $\Qvec=0$ is the
global minimizer of the bulk energy density, $f_B$, in this
high-temperature regime and it is not surprising that we observe a
dissipation of order in the interior. However, it is interesting
that there are no local fluctuations in the interior i.e. there
are no interior regions where $|\Qvec^*|$ experiences a local
increase compared to the boundary norm - $\max_{\rvec \in
\partial\Omega} |\Qvec_0(\rvec)|$. Therefore, Theorem~\ref{thm:2}
suggests a monotonic decrease in order as we move away from the
boundary and it would be interesting to analytically estimate the
characteristic length scale of order decay for this problem.

Our methods readily extend to a more general bulk energy density,
$f_{B,n}$, which is a polynomial of even degree $`n'$ in the
$\Qvec$-tensor invariants i.e. $\textrm{tr}\Qvec^2$ and
$\textrm{tr}\Qvec^3$ \cite{newtonmottram} with $n\geq 4$ (since
$f_{B,n}$ has to be minimally quartic to allow a first-order
nematic-isotropic phase transition). We take $f_{B,n}$ to be
\begin{equation}
\label{eq:newfb1} f_{B,n} = a_2(T) \textrm{tr}\Qvec^2 -
a_3\textrm{tr}\Qvec^3 + a_4\left(\textrm{tr}\Qvec^2\right)^2 +
\ldots + \sum_{m,p \in \mathbb{Z}^+; 2m+3p=n}
a_{m,p}\left(\textrm{tr}\Qvec^2\right)^m\left(\textrm{tr}\Qvec^3\right)^p
\end{equation}
where $\mathbb{Z}^+$ denotes the set of non-negative integers,
$a_3, a_4 > 0$ and
\begin{equation}
\label{eq:newfb2} a_{\frac{n}{2},0} > \sum_{m,p \in
\mathbb{Z}^+;p\geq 1; 2m+3p=n} |a_{m,p}|.
\end{equation} The first coefficient $a_2(T)$ has a linear
dependence on the absolute temperature by analogy with
(\ref{eq:a}) whereas the remaining coefficients
$\left\{a_3,a_4,\ldots, \left\{a_{m,p}\right\}\right\}$ are taken
to be temperature-independent, material-dependent bulk constants.

We define the corresponding Landau-De Gennes energy functional to
be
\begin{equation}
\label{eq:newfb3} I_{n}[\Qvec] = \int_{\Omega} f_{B,n}(\Qvec) +
L|\grad \Qvec|^2~dV
\end{equation}
and our admissible space is
\begin{equation}
\label{eq:newfb4} \Acal_n = \left\{\Qvec\in
W^{1,n}\left(\Omega;S_0\right); \Qvec = \Qvec_0 ~\textrm{on
$\partial \Omega$}\right\}
\end{equation}
where the Sobolev space $W^{1,n}$ is defined to be \cite{evans}
\begin{equation}
\label{eq:newfb5} W^{1,n}\left(\Omega;S_0\right) = \left\{\Qvec
\in S_0;~ \int_{\Omega} |\Qvec|^n + |\grad \Qvec|^n~dV < \infty
\right\}
\end{equation} and $\Qvec_0$ is a smooth, physically realistic
boundary condition in the sense of (\ref{eq:bc}). We have the
following result by analogy with Theorem~\ref{thm:1}.

\begin{prop}
\label{prop:fb} Let $\Qvec^*$ be a global minimizer of $I_n$ in
the admissible space $\Acal_n$. Then
\begin{equation}
\label{eq:newfb6} |\Qvec^*| \leq \max
\left\{C\left(a_2,a_3,a_4,\ldots,\left\{a_{m,p}\right\}\right),
~\max_{\rvec\in\partial\Omega}|\Qvec_0|\right\} \quad \textrm{on
$\bar{\Omega}$}
\end{equation}
where $C$ is a positive constant that only depends on the absolute
temperature and the bulk coefficients and is independent of the
elastic constant $L$ in (\ref{eq:newfb3}).
\end{prop}

\begin{proof}
The existence of a global energy minimizer, $\Qvec^*$, for $I_n$
in the space $\Acal_n$, follows from the direct methods in the
calculus of variations and we can use standard results in elliptic
regularity to deduce that $\Qvec^*$ is smooth everywhere in $\Omega$. Consider the function $|\Qvec^*|^2:\Omega
\to \Rr^+$ and assume that it attains its maximum at the interior
point $\rvec^*\in\Omega$.

The global minimizer $\Qvec^*$ is a classical solution of the
Euler-Lagrange equations
\begin{equation}
\label{eq:newfb7} 2L\Qvec_{ij,kk} = \frac{\partial
f_{B,n}}{\partial \Qvec_{ij}} - \frac{1}{3}\frac{\partial
f_{B,n}}{\partial \Qvec_{kk}}\delta_{ij}
\end{equation}
where the second term is a Lagrange multiplier accounting for
tracelessness. We multiply both sides of (\ref{eq:newfb7}) by
$\Qvec_{ij}$ and use $\Delta|\Qvec|^2=2  \left(|\grad\Qvec|^2 +
\Qvec_{ij}\Qvec_{ij,kk}\right)$ to get
\begin{equation}
\label{eq:newfb8} L \Delta |\Qvec|^2 = \Qvec_{ij} \frac{\partial
f_{B,n}}{\partial \Qvec_{ij}} + 2L|\grad \Qvec|^2.
\end{equation}
Then
\begin{equation}
\label{eq:newfb9} \Delta |\Qvec^*|^2 \leq 0 ~\textrm{at $\rvec^*
\in \Omega$}
\end{equation} from our hypothesis.

We note that $\Qvec_{ij} \frac{\partial f_{B,n}}{\partial
\Qvec_{ij}}$ is a polynomial of degree $`n'$ in $\Qvec$ and from
Lemma~\ref{lem:1} and (\ref{eq:newfb2}), we have that
$$\Qvec_{ij} \frac{\partial f_{B,n}}{\partial \Qvec_{ij}}\geq
K(|\Qvec|)$$ where
\begin{equation}
\label{eq:newfb10} K(|\Qvec|) = a_2|\Qvec|^2 -
\frac{a_3}{\sqrt{6}}|\Qvec|^3+ \ldots + \left(a_{\frac{n}{2},0} -
\sum_{m,p \in \mathbb{Z}^+;p\geq 1; 2m+3p=n}
|a_{m,p}|\right)|\Qvec|^n.
\end{equation} The function $K:\Omega \to \Rr$ is a polynomial
of degree $`n'$, has $`n'$ zeros
$\left\{|\Qvec_1|,|\Qvec_n|,\ldots,|\Qvec_n|\right\}$ where
$|\Qvec_1|\leq |\Qvec_2|\leq \ldots \leq |\Qvec_n|$ and $K$ is a
monotonically increasing function of $|\Qvec|$ for
$|\Qvec|>|\Qvec_n|$.

If $|\Qvec^*(\rvec^*)| > |\Qvec_n|$, then we necessarily have that
$\Delta |\Qvec^*|^2 > 0$ at $\rvec^*$ (from (\ref{eq:newfb8})),
contradicting the hypothesis (\ref{eq:newfb9}). We, thus, conclude
that
\begin{equation}
\label{eq:newfb11} |\Qvec^*| \leq |\Qvec_n| \quad \textrm{on
$\Omega$}
\end{equation}
where $|\Qvec_n| =
C\left(a_2,a_3,a_4,\ldots,\left\{a_{m,p}\right\}\right)$ can be
explicitly expressed in terms of the bulk coefficients.
Proposition~\ref{prop:fb} now follows from combining
(\ref{eq:newfb11}) and the maximum norm of $\Qvec_0$ on the
boundary.
\end{proof}

The explicit upper bound (\ref{eq:newfb6}) allows us to define the
admissible domain for the equilibrium scalar order
parameters as in Section~\ref{sec:oc}. For certain choices of the
bulk coefficients, this domain is larger than
$T_{\psi}$ and consequently, the equilibrium scalar order parameters may take
values outside the physical triangle.

\section{Discussion}
\label{sec:4}

We have studied qualitative properties of global minimizers of the
Landau-De Gennes energy functional, $I_{LG}$, in smooth
three-dimensional geometries with Dirichlet boundary conditions.
We have obtained an explicit upper bound for the norm of a global
energy minimizer in terms of the temperature and
material-dependent bulk constants, independent of the elastic
constant. In particular, we have defined two triangles in the
order-parameter, $(s,~r)$-plane: (a) the bulk triangle
$\triangle(T)$ which accounts for the stationary points of $f_B$
(b) the elastic triangle $\triangle_{el}(T)$ which accounts for
the effects of the elastic energy density and $\triangle(T)\subset
\triangle_{el}(T)$. The equilibrium scalar order parameters take
values inside or on the boundary of $\triangle_{el}(T)$ and the
distance $D(T)$ between $\triangle_{el}(T)$ and $\triangle(T)$
scales as
\begin{equation}
\label{eq:dis1} D(T) \leq s_+ = \frac{b + \sqrt{b^2-24ac}}{4c}.
\end{equation} This, in effect, quantifies the effect of elastic
perturbations on the bulk energy minimizers. Secondly, this
explicit bound is also compared to the probabilistic bounds in
(\ref{eq:7}) and we find that the equilibrium scalar order
parameters may move outside the physical triangle $T_\psi$ and take physically unrealistic values larger than unity, in the
low-temperature regime. For the liquid crystal material MBBA, the
Landau-De Gennes predictions fail to be physically realistic
within a $2^{o}C$-neighbourhood of the nematic-isotropic
transition temperature.

A natural question is - how can we reconcile the differences
between the Landau-De Gennes predictions and the probabilistic
second-moment definition of $\Qvec$ in the low-temperature regime?
The Landau-De Gennes bulk energy density, $f_B$, has no term that
enforces the probabilistic bounds in (\ref{eq:7}) or penalizes
configurations that lie outside $T_\psi$. A first-step in this
direction is to use a Ginzburg-Landau approach \cite{bbh}. We
define a modified Landau-De Gennes energy functional, $F_\eps$, as
shown below
\begin{equation}
\label{eq:dis2} F_\eps[\Qvec] = \int_{\Omega}
f_B\left(\Qvec\right) + f_\eps(|\Qvec|) + L|\grad \Qvec|^2~dV
\end{equation} where
\begin{equation}
  \label{eq:dis2new}
  f_\eps(|\Qvec|) =
  \begin{cases}
    0, ~ |\Qvec|\leq \frac{1}{\sqrt{6}},\\
      \\
 \frac{1}{\eps^2}\left(|\Qvec|^2 - \frac{1}{6}\right)^2, ~
 |\Qvec|>\frac{1}{\sqrt{6}},
  \end{cases}
\end{equation}
$f_B$ is as in (\ref{eq:13}) and $\eps>0$ is a small positive
parameter. We can obtain an explicit upper bound for the norm of a
global energy minimizer using a maximum principle approach as
shown below.

\begin{prop}
\label{prop:ginzburglandau} Let $\Qvec^*$ be a global minimizer of
$F_\eps$ in the admissible space $\Acal$ in (\ref{eq:31}). Then
\begin{equation}
\label{eq:dis3} |\Qvec^*|\leq \max\left\{\frac{1}{\sqrt{6}},~\frac{b\eps^2}{\sqrt{6}(8
+ 2\eps^2 c)} + \frac{\sqrt{64 + 16\eps^2\left(c-6a\right) +
\eps^4\left(b^2 - 24ac\right)}}{\sqrt{6}(8 + 2\eps^2 c)},
\max_{\rvec\in\partial\Omega}|\Qvec_0|\right\} \quad \textrm{on
$\bar{\Omega}$}
\end{equation} where $a,b,c$ are the bulk constants in $f_B$ in
(\ref{eq:13}).\end{prop}

\begin{proof}
The existence of a global energy minimizer $\Qvec^*$ follows from
the direct methods in the calculus of variations and we use
standard arguments in elliptic regularity to deduce that $\Qvec^*$
is smooth everywhere in $\Omega$, up to the boundary. Then
$|\Qvec^*|^2:\Omega \to \Rr^+$ is a smooth function and we assume
that it attains its maximum at an interior point $\rvec^*\in
\Omega$ and $|\Qvec^*(\rvec^*)| > \frac{1}{\sqrt{6}}$ (i.e.
$f_\eps \neq 0$ at $\rvec^*$ from (\ref{eq:dis2new})). It follows
that
$$\Delta |\Qvec^*|^2 \leq 0 \quad \textrm{at $\rvec^* \in\
\Omega$.}$$

The global minimizer $\Qvec^*$ is a classical solution of the
Euler-Lagrange equations
\begin{equation}
\label{eq:dis4} 2 L\Delta \Qvec_{\alpha\beta} = a
\Qvec_{\alpha\beta} + b\left(\frac{1}{3}\textrm{tr}\Qvec^2
\delta_{\alpha\beta} - \Qvec_{\alpha p}\Qvec_{p \beta}\right) +
c\left(\Qvec_{pq}\right)^2\Qvec_{\alpha\beta} + \frac{4
\Qvec_{\alpha\beta}}{\eps^2}\left(|\Qvec|^2 - \frac{1}{6}\right).
\end{equation}
Repeating the same arguments as in the proof of
Theorem~\ref{thm:2}, we have the following inequality
$$ L\Delta |\Qvec|^2 \geq M(|\Qvec|) = a |\Qvec|^2 -
\frac{b}{\sqrt{6}}|\Qvec|^3 + c|\Qvec|^4 +
\frac{4|\Qvec|^2}{\eps^2}\left(|\Qvec|^2 - \frac{1}{6}\right).
$$

We study the function $M:\Omega \to \Rr$ above. This function
has precisely three zeros:
\begin{eqnarray}
\label{eq:dis5} && |\Qvec| = 0 \nonumber \\
&& |\Qvec| = \frac{b\eps^2}{\sqrt{6}(8 + 2\eps^2 c)} \pm
\frac{\sqrt{64 + 16\eps^2\left(c-6a\right) + \eps^4\left(b^2 -
24ac\right)}}{\sqrt{6}(8 + 2\eps^2 c)}
\end{eqnarray}
and $M(|\Qvec|)>0$ for $|\Qvec|>\frac{b\eps^2}{\sqrt{6}(8 +
2\eps^2 c)} + \frac{\sqrt{64 + 16\eps^2\left(c-6a\right) +
\eps^4\left(b^2 - 24ac\right)}}{\sqrt{6}(8 + 2\eps^2 c)}$.
Therefore, we must have
$$|\Qvec^*(\rvec^*)|\leq \frac{b\eps^2}{\sqrt{6}(8
+ 2\eps^2 c)} + \frac{\sqrt{64 + 16\eps^2\left(c-6a\right) +
\eps^4\left(b^2 - 24ac\right)}}{\sqrt{6}(8 + 2\eps^2 c)}$$ in
order to have $\Delta|\Qvec^*|^2 \leq 0$ at $\rvec^*\in\Omega$.

We, thus, conclude that
\begin{equation}
\label{eq:dis6} |\Qvec^*| \leq
\max\left\{\frac{1}{\sqrt{6}},\frac{b\eps^2}{\sqrt{6}(8 + 2\eps^2 c)} +
\frac{\sqrt{64 + 16\eps^2\left(c-6a\right) + \eps^4\left(b^2 -
24ac\right)}}{\sqrt{6}(8 + 2\eps^2 c)},
\max_{\rvec\in\partial\Omega}|\Qvec_0|\right\} \quad \textrm{on
$\bar{\Omega}$}
\end{equation}
where the second term accounts for the maximum of $\left|\Qvec_0\right|$ on
$\partial\Omega$.
\end{proof}

In the limit $\eps \to 0$ and for a physically realistic boundary
condition $\Qvec_0$ satisfying (\ref{eq:bc}), the upper bound
(\ref{eq:dis3}) reduces to
\begin{equation}
\label{eq:dis7} |\Qvec^*|\leq \frac{1}{\sqrt{6}} + O(\eps) \quad
\textrm{on $\bar{\Omega}$}
\end{equation}
and consequently, the equilibrium scalar order parameters take
values within $T_\psi$ for all temperature regimes. Further, one
can show that the stationary points of the modified bulk energy
density $$f_{B,\eps}(\Qvec) = f_B(\Qvec) + f_\eps(|\Qvec|)
$$ are either isotropic or uniaxial $\Qvec$-tensors as in Proposition~\ref{prop:uniaxial}
and $f_{B,\eps}$ also predicts a first-order nematic-isotropic
phase transition. Therefore, the modified Landau-De Gennes energy
functional $F_{\eps}$ reproduces all the qualitative features of
$I_{LG}$ in (\ref{eq:12}), whilst respecting the probabilistic
bounds (\ref{eq:7}) in the limit $\eps \to 0^+$, for all
temperature regimes.

However, the Ginzburg-Landau approach in (\ref{eq:dis2}) does not
contain any information about the probabilistic second-moment
definition of $\Qvec$ in (\ref{eq:5}). A more systematic approach
is given in \cite{luckhurst} where they define a modified bulk
energy density $\Psi_B$ from the Maier-Saupe free energy $I_{MS}$
in (\ref{eq:maier-saupe}).
\begin{equation}
\label{eq:dis12} \Psi_B(\Qvec) = \inf_{\psi \in
\Acal_\Qvec}\int_{S^2}\psi(\pvec)\log \psi(\pvec)~d\pvec -
\frac{1}{2}U(T)|\Qvec|^2
\end{equation}
where
\begin{equation}
\label{eq:dis13} \Acal_\Qvec = \left\{\psi \in
L^1\left(S^2;\Rr^+\right);~ \Qvec =
\int_{S^2}\left(\pvec\otimes\pvec -
\frac{1}{3}\mathbf{I}\right)\psi(\pvec)~d\pvec \right\}
\end{equation}
is the space of all probability distribution functions $\psi$ that
have a fixed normalized second moment $\Qvec$, as in (\ref{eq:5})\begin{footnote}{$\psi \in L^1\left(S^2;\Rr^+\right)$ simply means that $\int_{S^2}\psi(\pvec)~d\pvec$ is well-defined}\end{footnote}.

The first term in (\ref{eq:dis12}) is the entropy contribution,
where we minimize the integral over all probability distribution
functions that have a fixed second moment $\Qvec$. This term
diverges whenever the probabilistic bounds in (\ref{eq:7}) are
violated and enforces the equilibrium scalar order parameters to
lie strictly inside $T_\psi$. The second term in (\ref{eq:dis12})
is simply the Maier-Saupe interaction energy. The uniaxial case is
treated in \cite{luckhurst}. We plan to study the biaxial case and
include spatial inhomogeneities into this model. This will be
reported in future work \cite{jmbam}.

\section*{Acknowledgements} A.~Majumdar was supported by a Royal
Commission for the Exhibition of 1851 Research Fellowship till
October 2008. This publication is now based on work supported by
Award No. KUK-C1-013-04 , made by King Abdullah University of
Science and Technology (KAUST) to the Oxford Centre for
Collaborative Applied Mathematics. The author gratefully
acknowledges helpful comments and suggestions made by John Ball, Giovanni De Matteis, Geoffrey Luckhurst, 
Carlos Mora-Corral and Tim Sluckin.


\begin{thebibliography}{5}

\bibitem{ball} J.~M.~Ball, Graduate lecture course ``Mathematical theories of liquid crystals", 2007.
\bibitem{jmbam} J.~M.~Ball and A.~Majumdar, A novel approach to
continuum energy functionals for nematic liquid crystals, in preparation 2008.
\bibitem{bbh} F. Bethuel, H. Brezis  and F.H\'{e}lein,
Asymptotics for the minimization of a Ginzburg-Landau functional.
Calc. Var. Partial Differential Equations  1, no. 2, 123--148
(1993).
\bibitem{dacorogna} B.~Dacorogna, Direct methods in the calculus of variations. Applied Mathematical Sciences,
78, Springer, 1989.
\bibitem{gartland} T.~Davis and E.~Gartland, Finite element analysis of the Landau--De Gennes minimization problem for liquid crystals.
SIAM Journal of Numerical Analysis, 35, 336-362 (1998).
\bibitem{dg}
 P.~G.~De Gennes, The physics of liquid crystals. Oxford, Clarendon Press, 1974.
\bibitem{evans} L.~Evans, Partial Differential Equations.
American Mathematical Society, Providence, 1998.
 \bibitem{forest1} M.~G.~Forest, Q.~Wang and H.~Zhou, Homogeneous
pattern selection and director instabilities of nematic liquid
crystal polymers induced by elongational flows. Physics of Fluids,
12, no. 3, 490--498 (2000).
\bibitem{forest2} M.~G.~Forest, Q.~Wang and H.~Zhou, Exact banded patterns from a Doi-Marruci-Greco model of nematic liquid crystal polymers.
Physical Review E, 61, no. 6, 6655--6662 (2000).
\bibitem{gilbarg}  D.~Gilbarg and N.~Trudinger, Elliptic Partial Differential Equations of Second Order.
Grundlehren der mathematischen Wissenschaften, a Series of
Comprehensive Studies in Mathematics, Vol. 224, 2nd Edition,
Springer 1977.
\bibitem{luckhurst} J.~Katriel, G.F.~Kventsel, G.R.~Luckhurst and T.J.~Sluckin, Free Energies in the Landau and Molecular Field Approaches. 
Liquid Crystals 1, 337 -- 55 (1986).
\bibitem{kg} S.~Kitson and A.~Geisow,
Controllable alignment of nematic liquid crystals around
microscopic posts: Stabilization of multiple states. Applied
Physics Letters, 80, 3635 -- 3637 (2002).
\bibitem{Lin}
F.~H.~Lin and C.~Liu, Static and Dynamic Theories of Liquid
Crystals. Journal of Partial Differential Equations, 14, no. 4,
289--330  (2001).
\bibitem{ms}W.~Maier and A.~Saupe, A simple molecular
statistical theory of the nematic crystalline-liquid phase. I Z
Naturf. a 14, 882--889 (1959).
\bibitem{gartland2} S.~Mkaddem and E.~C.~Gartland, Fine structure of defects
in radial nematic droplets. Physical Review E, 62, no.5,
6694--6705 (2000).
\bibitem{newtonmottram} N.~J.~Mottram and C.~Newton, Introduction to Q-tensor
Theory. University of Strathclyde, Department of Mathematics,
Research Report, 10 (2004).
\bibitem{palffy} P.~Palffy-Muhoray, Lecture notes on Liquid Crystal Materials, 2004.
 \bibitem{straley}M.~J.~Stephen and  J.~P.~Straley, Physics of liquid crystals.
 Reviews of Modern Physics, 46, 617 -- 701  (1974).
\bibitem{virga} E.~G.~Virga, Variational theories for liquid crystals. Chapman and Hall, London, 1994.
\end{thebibliography}
\end{document}